\documentclass[12pt]{article}
\usepackage{amsmath,amssymb}

\newtheorem{theorem}{Theorem}[section]
\newtheorem{proposition}[theorem]{Proposition}
\newtheorem{corollary}[theorem]{Corollary}

\newtheorem{remark}[theorem]{Remark}

\newenvironment{proof}{\smallskip\par{\sc Proof.}\enspace}%
 {{\unskip\nobreak\hfil\penalty50\hskip2em
          \hbox{}\nobreak\hfil{\rule[-1pt]{5pt}{10pt}}
          \parfillskip=0pt\finalhyphendemerits=0
          \par\medskip}} 

\makeatletter
\def\section{\@startsection {section}{1}{\z@}{3.25ex plus 1ex minus
 .2ex}{1.5ex plus .2ex}{\large\bf}}
\def\subsection{\@startsection{subsection}{2}{\z@}{3.25ex plus 1ex minus
 .2ex}{1.5ex plus .2ex}{\normalsize\bf}}
\@addtoreset{equation}{section} 
\makeatother

\begin{document}

\begin{center}
\LARGE\textsf{A note on a local limit theorem for Wiener space valued random variables}
\end{center}

\vspace*{.3in}

\begin{center}
\sc
\large{Alberto Lanconelli\footnote{Dipartimento di Matematica, Universit\'a degli Studi di Bari Aldo Moro, Via E. Orabona 4, 70125 Bari - Italia. E-mail: \emph{alberto.lanconelli@uniba.it}}
\quad Aurel I. Stan\footnote{Department of Mathematics, Ohio State University at Marion, 1465 Mount Vernon Avenue, Marion, OH 43302, U.S.A. E-mail: {\em stan.7@osu.edu}}}\\
\end{center}

\vspace*{.3in}

\begin{abstract}
We prove a local limit theorem, i.e. a central limit theorem for densities, for a sequence of independent and identically distributed random variables taking values on an abstract Wiener space; the common law of those random variables is assumed to be absolutely continuous with respect to the reference Gaussian measure. We begin by showing that the key roles of scaling operator and convolution product in this infinite dimensional Gaussian framework are played  by the Ornstein-Uhlenbeck semigroup and Wick product, respectively. We proceed by establishing a necessary condition on the density of the random variables for the local limit theorem to hold true. We then reverse the implication and prove under an additional assumption the desired $\mathcal{L}^1$-convergence of the density of $\frac{X_1+\cdot\cdot\cdot+X_n}{\sqrt{n}}$.  We close the paper comparing our result with certain Berry-Esseen bounds for multidimensional central limit theorems.
\end{abstract}

\bigskip


\bigskip


\section{Introduction}
The classic one dimensional central limit theorem asserts that, for a given sequence $\{X_n\}_{n\geq 1}$ of independent and identically distributed random variables with mean zero and variance one, the sequence $\frac{X_1+\cdot\cdot\cdot+X_n}{\sqrt{n}}$ converges in distribution as $n\to +\infty$ to the standard normal law. One may wonder whether under more restrictive assumptions the previously mentioned convergence holds in some stronger sense. One can for instance be interested in the convergence of the density (with respect to the Lebesgue measure) of the law of $\frac{X_1+\cdot\cdot\cdot+X_n}{\sqrt{n}}$ towards the function $\frac{1}{\sqrt{2\pi}}e^{-\frac{x^2}{2}}$. This kind of results goes under the name of local limit theorem: Prokhorov (1952) established convergence in $\mathcal{L}^1$, Ranga Rao and Varadarajan (1960) obtained point-wise convergence while Gnedenko (1954) studied uniform convergence. Barron (1986) proved that the relative entropy (or Kullback-Leibler divergence) of $\frac{X_1+\cdot\cdot\cdot+X_n}{\sqrt{n}}$ with respect to the standard Gaussian measure tends to zero (monotonically along a certain subsequence). Infinite dimensional local limit theorems were considered by Bloznelis (2002), who provided a counterexample on the validity of Prokhorov's theorem on general Hilbert spaces, and by Davydov (1992) who suggested a variant of an infinite dimensional local limit theorem.    \\
The aim of the present note is to prove a local limit theorem for sequences of independent and identically distributed random variables taking values on an abstract Wiener space. The main novelty of our result consists in utilizing some notions and techniques from stochastic analysis as the infinite dimensional counterpart of the basic tools adopted to treat the finite dimensional case. In fact, we show in Proposition \ref{wick-convolution} below that the Wick product and the Ornstein-Uhlenbeck semigroup are the natural convolution product and scaling operator for densities in our infinite dimensional Gaussian setting, respectively. Then, by means of the results obtained by Da Pelo, et al. (2011) and the Nelson hyper-contractivity theorem we obtain under certain conditions the desired local limit theorem with an explicit rate of convergence. As a by-product of our method we obtain a dimension independent Berry-Essen bound for a large class of multivariate probability distributions.\\
The paper is organized as follows: Section 2 collects briefly notation and background material from the analysis on infinite dimensional Gaussian spaces while in Section 3 after few preparatory results and observations we state and prove our main theorem (see Theorem \ref{main theorem} below) followed by a detailed inspection of the finite dimensional case.

\section{Framework}
The aim of this section is to collect the necessary background
material and fix the notation. For more details the interested reader is
referred to the books of Bogachev (1998), Janson (1997) and Nualart (2006) .\\
Let $(H,W,\mu)$ be an \emph{abstract Wiener space}, that means
$(H,\langle\cdot,\cdot\rangle_H)$ is a separable Hilbert space which
is continuously and densely embedded in the Banach space
$(W,\Vert\cdot\Vert_W)$ and $\mu$ is a Gaussian probability measure
on the Borel sets $\mathcal{B}(W)$ of $W$ such that
\begin{eqnarray}\label{Gaussian characteristic}
\int_{W}e^{i\langle w,w^*\rangle}d\mu(w)=e^{-\frac{1}{2}\Vert
w^*\Vert_H^2},\quad\mbox{ for all }w^*\in W^*.
\end{eqnarray}
Here $W^*\subset H$ denotes the dual space of $W$, which is dense in $H$, and $\langle\cdot,\cdot\rangle$ stands for the dual
pairing between $W$ and $W^*$. We will refer to $H$ as the
\emph{Cameron-Martin} space of $W$. Set for $p\geq 1$
\begin{eqnarray*}
\mathcal{L}^p(W,\mu):=\Big\{f:W\to\mathbb{R}\mbox{ such that }\Vert
f\Vert_p:=\Big(\int_W|f(w)|^pd\mu(w)\Big)^{\frac{1}{p}}<+\infty\Big\}.
\end{eqnarray*}
It follows from
(\ref{Gaussian characteristic}) that the map
\begin{eqnarray*}
W^*&\to&\mathcal{L}^2(W,\mu)\\
w^*&\mapsto&\langle w,w^*\rangle
\end{eqnarray*}
is an isometry; we can therefore define for $\mu$-almost all $w\in
W$ the quantity $\langle w,\mathfrak{h}\rangle$ for $\mathfrak{h}\in H$  as an element of
$\mathcal{L}^2(W,\mu)$. This element will be denoted
by $\delta(\mathfrak{h})$.\\
Recall that by the Wiener-It\^o chaos
decomposition theorem any element $f$ in $\mathcal{L}^2(W,\mu)$ has
an infinite orthogonal expansion
\begin{eqnarray*}
f=\sum_{n\geq 0} \delta^n(\mathfrak{h}_n)
\end{eqnarray*}
where $\mathfrak{h}_n\in H^{\hat{\otimes}n}$, the space of symmetric elements
of $H^{\otimes n}$, and $\delta^n(\mathfrak{h}_n)$ stands for the multiple
Wiener-It\^o integral of $\mathfrak{h}_n$. For each $n\geq 0$ denote by $J_n$ the orthogonal projection onto the $n$-th Wiener chaos, i.e. for $f\in\mathcal{L}^2(W,\mu)$ with chaos expansion
$\sum_{n\geq 0} \delta^n(\mathfrak{h}_n)$, one has $J_n(f)=\delta^n(\mathfrak{h}_n)$. It is worth to mention that for any $p>1$ the operators $J_n$ can be extended to continuous linear operators from $\mathcal{L}^p(W,\mu)$ into itself.\\
For any $|\lambda|\leq 1$ define the operator $\Gamma(\lambda)$ acting
on $\mathcal{L}^2(W,\mu)$ as
\begin{eqnarray*}
\Gamma(\lambda)\Big(\sum_{n\geq 0}\delta^n(\mathfrak{h}_n)\Big):=\sum_{n\geq 0}
\lambda^n\delta^n(\mathfrak{h}_n).
\end{eqnarray*}
Observe that $\Gamma(\lambda)$ coincides with the
Ornstein-Uhlenbeck semigroup
\begin{eqnarray}\label{OU}
(P_tf)(w):=\int_Wf(e^{-t}w+\sqrt{1-e^{-2t}}\tilde{w})d\mu(\tilde{w}),\quad
w\in W, t\geq 0
\end{eqnarray}
(take $\lambda=e^{-t}$) and therefore it can be extended to a continuous linear operator on $\mathcal{L}^p(W,\mu)$ for every $p\geq 1$. One of the crucial features of the operator $\Gamma(\lambda)$ is the hyper-contractive property proved in the celebrated Nelson theorem (Nelson (1973)): for any $1\leq p\leq q\leq +\infty$ and $|\lambda|\leq \sqrt{\frac{p-1}{q-1}}$ one has the inequality
\begin{eqnarray}\label{nelson}
\Vert\Gamma(\lambda)f\Vert_q\leq \Vert f\Vert_p,\quad f\in\mathcal{L}^p(W,\mu).  
\end{eqnarray}
We also mention the useful property $\Gamma(\lambda_2)\Gamma(\lambda_1)=\Gamma(\lambda_2\cdot\lambda_1)$ which is equivalent to the semigroup property of (\ref{OU}).\\
We now define the Wick product: for $\mathfrak{h},\mathfrak{k}\in H$ set
\begin{eqnarray*}
\mathcal{E}(\mathfrak{h})\diamond\mathcal{E}(\mathfrak{k}):=\mathcal{E}(\mathfrak{h}+\mathfrak{k}).
\end{eqnarray*}
where 
\begin{eqnarray*}
\mathcal{E}(\mathfrak{h}):=\exp\Big\{\delta(\mathfrak{h})-\frac{1}{2}\Vert \mathfrak{h}\Vert^2_H\Big\}.
\end{eqnarray*}
This is called the \emph{Wick product} of $\mathcal{E}(\mathfrak{h})$ and
$\mathcal{E}(\mathfrak{k})$. Extend this operation by linearity to the linear span of the $\mathcal{E}(\mathfrak{h})$'s (which is dense in all the $\mathcal{L}^p(W,\mu)$'s) 
to get a commutative, associative and distributive (with respect to the sum)
multiplication. The Wick product is easily seen to be an unbounded bilinear form on the $\mathcal{L}^p(W,\mu)$ spaces; nevertheless, by applying the operator $\Gamma(\lambda)$ one obtains
\begin{eqnarray}\label{HYL}
\Vert\Gamma(\alpha)f\diamond\Gamma(\beta)g\Vert_p\le\Vert f\Vert_p\Vert g\Vert_p
\end{eqnarray}
for all $|\alpha|,|\beta|\leq 1$ with $\alpha^2+\beta^2\leq 1$ (see Da Pelo, et al. (2011)).\\
If the Wick product $f\diamond g$ of $f,g\in\mathcal{L}^p(W,\mu)$, $p>1$ exists in $\mathcal{L}^p(W,\mu)$, then for any $\mathfrak{h}\in H$ one has
\begin{eqnarray}\label{charact}
\int_W(f\diamond g)(w)\mathcal{E}(\mathfrak{h})(w)d\mu(w)=\int_Wf(w)\mathcal{E}(\mathfrak{h})(w)d\mu(w)\cdot\int_Wg(w)\mathcal{E}(\mathfrak{h})(w)d\mu(w).
\end{eqnarray}
In particular for $\mathfrak{h}=0$ one gets
\begin{eqnarray*}
\int_W(f\diamond g)(w)d\mu(w)=\int_Wf(w)d\mu(w)\cdot\int_Wg(w)d\mu(w).
\end{eqnarray*}
To conclude we mention the useful functorial behavior of $\Gamma(\lambda)$ with respect to the Wick product $\diamond$:
\begin{eqnarray}\label{functor}
\Gamma(\lambda)(f\diamond g)=\Gamma(\lambda)f\diamond\Gamma(\lambda)g.
\end{eqnarray} 
For additional information on the Wick product the reader is referred to the book of Holden, et al. (2009), the papers Da Pelo, et al. (2011), Da Pelo, et al. (2013) and the references quoted there.

\section{Main result}

In this section we are going to state and prove a local limit theorem for a sequence of independent and identically distributed random variables taking values on an abstract Wiener space. The next proposition  tells that in Gaussian spaces the role of the convolution product between functions is played by the Wick product. Similar results can be obtained for the chi-squared distribution (see Lanconelli and Sportelli (2012)) and the Poisson distribution (see Lanconelli and Stan (2013)).
\begin{proposition}\label{wick-convolution}
Let $X_1,...,X_n$ be independent random variables defined on a probability space $(\Omega,\mathcal{F},\mathcal{P})$ and taking values on the abstract Wiener space $(W,\mathcal{B}(W),\mu)$. Assume that for each $j\in\{1,...,n\}$ the law of $X_j$  is absolutely continuous with respect to the measure $\mu$ and denote its density by $f_j$. Choose $\alpha_1,...,\alpha_n\in [-1,1]$ such that $\sum_{j=1}^n\alpha_j^2=1$. Then the law of
$\alpha_1 X_1+\cdot\cdot\cdot+\alpha_n X_n$ is also absolutely continuous
with respect to $\mu$ with density given by
\begin{eqnarray*}
\Gamma({\alpha_1})f_1\diamond\cdot\cdot\cdot\diamond\Gamma({\alpha_n})f_n.
\end{eqnarray*}
\end{proposition}

\begin{proof}
We start computing the Fourier transform of the law of $\alpha_1 X_1+\cdot\cdot\cdot+\alpha_n X_n$. Fix $w^*\in W^*$; then from the assumption of independence we get
\begin{eqnarray*}
E[\exp\{i\langle\alpha_1 X_1+\cdot\cdot\cdot+\alpha_n X_n,w^*\rangle\}]&=&E[\exp\{i\langle\alpha_1 X_1,w^*\rangle\}\cdot\cdot\cdot\exp\{i\langle\alpha_n X_n,w^*\rangle\}]\\
&=&E[\exp\{i\langle\alpha_1 X_1,w^*\rangle\}]\cdot\cdot\cdot E[\exp\{i\langle\alpha_n X_n,w^*\rangle\}]
\end{eqnarray*}
Observe that
\begin{eqnarray*}
E[\exp\{i\langle\alpha_j X_j,w^*\rangle\}]&=&\int_{W}e^{i\alpha_j\langle w,w^*\rangle}f_j(w)d\mu(w)\\
&=&e^{-\frac{\alpha_j^2}{2}\Vert w^*\Vert^2_H}\int_{W}e^{i\alpha_j\langle w,w^*\rangle+\frac{\alpha_j^2}{2}\Vert w^*\Vert^2_H}f_j(w)d\mu(w)\\
&=&e^{-\frac{\alpha_j^2}{2}\Vert w^*\Vert^2_H}\int_{W}\Gamma(\alpha_j)\Big(e^{i\langle w,w^*\rangle+\frac{1}{2}\Vert w^*\Vert^2_H}\Big)f_j(w)d\mu(w)\\
&=&e^{-\frac{\alpha_1^2}{2}\Vert w^*\Vert^2_H}\int_{W}e^{i\langle w,w^*\rangle+\frac{1}{2}\Vert w^*\Vert^2_H}(\Gamma(\alpha_j)f_j)(w)d\mu(w).
\end{eqnarray*}
Here we used the identity $\Gamma(\lambda)\mathcal{E}(\mathfrak{h})=\mathcal{E}(\lambda\mathfrak{h})$ and the self-adjointness of $\Gamma(\lambda)$. Therefore,
\begin{eqnarray*}
&&E[\exp\{i\langle\alpha_1 X_1+\cdot\cdot\cdot+\alpha_n X_n,w^*\rangle\}]\\
&=&E[\exp\{i\langle\alpha_1 X_1,w^*\rangle\}]\cdot\cdot\cdot E[\exp\{i\langle\alpha_n X_n,w^*\rangle\}]\\
&=&\prod_{j=1}^ne^{-\frac{\alpha_j^2}{2}\Vert w^*\Vert^2_H}\int_{W}e^{i\langle w,w^*\rangle+\frac{1}{2}\Vert w^*\Vert^2_H}(\Gamma(\alpha_j)f_j)(w)d\mu(w)\\
&=&e^{-\frac{1}{2}\Vert w^*\Vert^2_H}\prod_{j=1}^n\int_{W}e^{i\langle w,w^*\rangle+\frac{1}{2}\Vert w^*\Vert^2_H}(\Gamma(\alpha_j)f_j)(w)d\mu(w)\\
&=&e^{-\frac{1}{2}\Vert w^*\Vert^2_H}\int_{W}e^{i\langle w,w^*\rangle+\frac{1}{2}\Vert w^*\Vert^2_H}(\Gamma(\alpha_1)f_1\diamond\cdot\cdot\cdot\diamond\Gamma(\alpha_n)f_n)(w)d\mu(w)\\
&=&\int_{W}e^{i\langle w,w^*\rangle}(\Gamma(\alpha_1)f_1\diamond\cdot\cdot\cdot\diamond\Gamma(\alpha_n)f_n)(w)d\mu(w),
\end{eqnarray*}
where in the third equality we used the assumption $\sum_{j=1}^n\alpha_j^2=1$ while in the fourth equality we utilized the characterizing property of the Wick product (\ref{charact}). To sum up, we obtained the identity
\begin{eqnarray*}
E[\exp\{i\langle\alpha_1 X_1+\cdot\cdot\cdot+\alpha_n X_n,w^*\rangle\}]&=&\int_{W}e^{i\langle w,w^*\rangle}(\Gamma(\alpha_1)f_1\diamond\cdot\cdot\cdot\diamond\Gamma(\alpha_n)f_n)(w)d\mu(w)
\end{eqnarray*} 
which is precisely what we wanted to prove.
\end{proof}

We are now ready to treat our local limit theorem. We begin by providing a necessary condition which corresponds, from the point of view of the chaos decomposition, to the usual assumption of the classic central limit theorem. To illustrate this point, consider the following simple situation:\\
Suppose that the law $\nu$ of a real valued random variable $X$ is absolutely continuous with respect to the one dimensional standard Gaussian measure $\mu$. Denote by $f$ the density of $\nu$ with respect to $\mu$ and assume that $f\in\mathcal{L}^2(\mathbb{R},\mu)$. It is well known that the monic Hermite polynomials $\{h_n\}_{n\geq 0}$ constitute an orthogonal basis for $\mathcal{L}^2(\mathbb{R},\mu)$; one can therefore write
\begin{eqnarray*}
f(x)=\sum_{n\geq 0}a_nh_n(x),\quad a_n\in\mathbb{R}.
\end{eqnarray*}
Since $h_0(x)=1$, $h_1(x)=x$ and $h_2(x)=x^2-1$, if $X$ has mean zero and unit variance, one deduces that
\begin{eqnarray*}
a_0&=&\int_{\mathbb{R}}f(x)d\mu(x)=\int_{\mathbb{R}}d\nu(x)=1\\
a_1&=&\int_{\mathbb{R}}xf(x)d\mu(x)=\int_{\mathbb{R}}xd\nu(x)=E[X]=0\\
a_2&=&\frac{1}{2}\int_{\mathbb{R}}(x^2-1)f(x)d\mu(x)=\frac{1}{2}\Big(\int_{\mathbb{R}}x^2d\nu(x)-1\Big)=\frac{1}{2}(Var(X)-1)=0.\\
\end{eqnarray*}
Therefore the assumptions of the central limit theorem, i.e. mean zero and unit variance, imply that $f$ has to take the form
\begin{eqnarray*}
f(x)=1+\sum_{n\geq 3}a_nh_n(x).
\end{eqnarray*}

\begin{proposition}
Let $\{X_n\}_{n\geq 1}$ be a sequence of independent and identically distributed  random variables defined on a probability space $(\Omega,\mathcal{F},\mathcal{P})$ and taking values on the abstract Wiener space $(W,\mathcal{B}(W),\mu)$. Suppose that the common law of the $X_n$'s is absolutely continuous with respect to the measure $\mu$ and denote its density by $f$. Assume that $f\in\mathcal{L}^p(W,\mu)$ for some $p>1$.\\
If the density of $\frac{X_1+\cdot\cdot\cdot+X_n}{\sqrt{n}}$ converges to $1$ in $\mathcal{L}^1(W,\mu)$ as $n$ tends to infinity, then $f$ must be orthogonal to the first and second chaoses, i.e. $J_1f=0$ and $J_2f=0$.
\end{proposition}
\begin{proof}
According to Proposition \ref{wick-convolution} the density of  $\frac{X_1+\cdot\cdot\cdot+X_n}{\sqrt{n}}$ is given by
\begin{eqnarray*}
\Gamma\Big(\frac{1}{\sqrt{n}}\Big)f\diamond\cdot\cdot\cdot\diamond\Gamma\Big(\frac{1}{\sqrt{n}}\Big)f=\Big(\Gamma\Big(\frac{1}{\sqrt{n}}\Big)f\Big)^{\diamond n}.
\end{eqnarray*}
Assume that
\begin{eqnarray*}
\lim_{n\to +\infty}\Big\Vert\Big(\Gamma\Big(\frac{1}{\sqrt{n}}\Big)f\Big)^{\diamond n}-1\Big\Vert_1=0.
\end{eqnarray*}
This implies that for any $w^*\in W^*$
\begin{eqnarray*}
&&\lim_{n\to +\infty}\int_{W}\Big(\Gamma\Big(\frac{1}{\sqrt{n}}\Big)f\Big)^{\diamond n}(w)\exp\Big\{i\langle w,w^*\rangle+\frac{1}{2}\Vert w^*\Vert^2_H\Big\}d\mu(w)\\
&=&\int_{W}\exp\Big\{i\langle w,w^*\rangle+\frac{1}{2}\Vert w^*\Vert^2_H\Big\}d\mu(w)\\
&=& 1.
\end{eqnarray*}
On the other hand,
\begin{eqnarray*}
&&\int_{W}\Big(\Gamma\Big(\frac{1}{\sqrt{n}}\Big)f\Big)^{\diamond n}(w)\exp\Big\{i\langle w,w^*\rangle+\frac{1}{2}\Vert w^*\Vert^2_H\Big\}d\mu(w)\\
&=&\Big(\int_{W}\Gamma\Big(\frac{1}{\sqrt{n}}\Big)f(w)\exp\Big\{i\langle w,w^*\rangle+\frac{1}{2}\Vert w^*\Vert^2_H\Big\}d\mu(w)\Big)^n\\
&=&\Big(\int_{W}\Gamma\Big(\frac{\sqrt{\gamma}}{\sqrt{n}}\Big)\Big(\Gamma\Big(\frac{1}{\sqrt{\gamma}}\Big)f\Big)(w)\exp\Big\{i\langle w,w^*\rangle+\frac{1}{2}\Vert w^*\Vert^2_H\Big\}d\mu(w)\Big)^n\\
&=&\Big(1+i\frac{\sqrt{\gamma}}{\sqrt{n}}\langle \mathfrak{h}_1,w^*\rangle_H-\frac{\gamma}{n}\langle \mathfrak{h}_2,(w^*)^{\otimes 2}\rangle_{H^{\otimes 2}}+o\Big(\frac{1}{n}\Big)\Big)^n
\end{eqnarray*}
where the $\mathfrak{h}_k$'s are the kernels in the Wiener-It\^o chaos expansion of $\Gamma\Big(\frac{1}{\sqrt{\gamma}}\Big)f$ and $\gamma\geq 1$ is chosen big enough to guarantee that $\Gamma\Big(\frac{1}{\sqrt{\gamma}}\Big)f\in\mathcal{L}^2(W,\mu)$ (this can be done via inequality (\ref{nelson})). 
The limit of the last expression is $1$, for all $w^*\in W^*$, provided that $\mathfrak{h}_1=0$ and $\mathfrak{h}_2=0$ which in turn implies the same condition on the kernels of $f$.
\end{proof}

The following is the main result of the present paper. It reverses under an additional smoothness condition  the implication of the previous proposition. 

\begin{theorem}\label{main theorem}
Let $\{X_n\}_{n\geq 1}$ be a sequence of independent and identically distributed  random variables defined on a probability space $(\Omega,\mathcal{F},\mathcal{P})$ and taking values on the abstract Wiener space $(W,\mathcal{B}(W),\mu)$. Suppose that the common law of the $X_n$'s is absolutely continuous with respect to the measure $\mu$ and with a density of the form $\Gamma(\sqrt{\alpha})f$ where $f$ is a non negative element of $\mathcal{L}^p(W,\mu)$ for some $p>1$ and $\alpha\in ]0,1[$.\\ 
If the density of the $X_n$'s is orthogonal to the first and the second Wiener chaoses, then the density of  
\begin{eqnarray}\label{limit}
\frac{X_1+\cdot\cdot\cdot+X_n}{\sqrt{n}}
\end{eqnarray}
converges in $\mathcal{L}^1(W,\mu)$ to $1$ as $n$ tends to infinity with rate of convergence of order $\frac{1}{\sqrt{n}}$. 
\end{theorem}

\begin{remark}
The function $f$ in the statement of the previous theorem is itself a density function: it is by assumption non negative and its integral over the whole space is, due to the identity 
\begin{eqnarray*}
\int_W(\Gamma(\sqrt{\alpha})f)(w)d\mu(w)=\int_W f(w)d\mu(w),
\end{eqnarray*}
equal to one.
\end{remark}

\begin{remark}
The assumption that $X_n$ has a density of the form $\Gamma(\sqrt{\alpha})f$ in the statement of Theorem \ref{main theorem} has a clear probabilistic meaning; in fact, using Proposition \ref{wick-convolution} and the identity
\begin{eqnarray*}
\Gamma(\sqrt{\alpha})f&=&\Gamma(\sqrt{\alpha})f\diamond\Gamma(\sqrt{1-\alpha})1
\end{eqnarray*}
we deduce that $X_n$ has a density of the form $\Gamma(\sqrt{\alpha})f$ if and only if the law of $X_n$ is equal to the one of $\sqrt{\alpha} X+\sqrt{1-\alpha}Z$ where the density of $X$ is $f$ and $Z$ is an independent Gaussian random variable with law $\mu$ (and hence unit density). This smoothness condition corresponds with the one required by Linnik (1973) in proving a finite dimensional information-theoretic central limit theorem.   
\end{remark}

\begin{proof}
From Proposition \ref{wick-convolution} the density of $\frac{X_1+\cdot\cdot\cdot+X_n}{\sqrt{n}}$ is given by
\begin{eqnarray*}
\Gamma\Big(\frac{\sqrt{\alpha}}{\sqrt{n}}\Big)f\diamond\cdot\cdot\cdot\diamond\Gamma\Big(\frac{\sqrt{\alpha}}{\sqrt{n}}\Big)f=\Big(\Gamma\Big(\frac{\sqrt{\alpha}}{\sqrt{n}}\Big)f\Big)^{\diamond n}.
\end{eqnarray*}
Observe that in view of (\ref{functor}) we can write without ambiguity the right hand side of the previous equation as 
$\Gamma\Big(\frac{\sqrt{\alpha}}{\sqrt{n}}\Big)f^{\diamond n}$.\\ 
Our aim is to prove that
\begin{eqnarray*}
\lim_{n\to +\infty}\Big\Vert \Gamma\Big(\frac{\sqrt{\alpha}}{\sqrt{n}}\Big)f^{\diamond n}-1\Big\Vert_1=0.
\end{eqnarray*}
First of all, exploiting the associativity and distributivity of the Wick product we write 
\begin{eqnarray*}
\Gamma\Big(\frac{\sqrt{\alpha}}{\sqrt{n}}\Big)f^{\diamond
n}-1&=&\sum_{j=1}^n\Gamma\Big(\frac{\sqrt{\alpha}}{\sqrt{n}}\Big)f^{\diamond
j}-\Gamma\Big(\frac{\sqrt{\alpha}}{\sqrt{n}}\Big)f^{\diamond j-1}\\
&=&\sum_{j=1}^n\Gamma\Big(\frac{\sqrt{\alpha}}{\sqrt{n}}\Big)f^{\diamond
j-1}\diamond\Big(\Gamma\Big(\frac{\sqrt{\alpha}}{\sqrt{n}}\Big)f-1\Big).
\end{eqnarray*}
Now take the $\mathcal{L}^1(W,\mu)$-norm and use the triangle
inequality:
\begin{eqnarray*}
\Big\Vert\Gamma\Big(\frac{\sqrt{\alpha}}{\sqrt{n}}\Big)f^{\diamond
n}-1\Big\Vert_1&=&\Big\Vert\sum_{j=1}^n\Gamma\Big(\frac{\sqrt{\alpha}}{\sqrt{n}}\Big)f^{\diamond
j-1}\diamond\Big(\Gamma\Big(\frac{\sqrt{\alpha}}{\sqrt{n}}\Big)f-1\Big)\Big\Vert_1\\
&\leq&\sum_{j=1}^n\Big\Vert\Gamma\Big(\frac{\sqrt{\alpha}}{\sqrt{n}}\Big)f^{\diamond
j-1}\diamond\Big(\Gamma\Big(\frac{\sqrt{\alpha}}{\sqrt{n}}\Big)f-1\Big)\Big\Vert_1.
\end{eqnarray*}
Apply inequality (\ref{HYL}) (actually we need only the $\mathcal{L}^1$-form of the inequality which was proven
before in the paper Lanconelli and Stan (2010)) to get
\begin{eqnarray*}
\Big\Vert\Gamma\Big(\frac{\sqrt{\alpha}}{\sqrt{n}}\Big)f^{\diamond
n}-1\Big\Vert_1&\leq&\sum_{j=1}^n\Big\Vert\Gamma\Big(\frac{\sqrt{\alpha}}{\sqrt{n}}\Big)f^{\diamond
j-1}\diamond\Big(\Gamma\Big(\frac{\sqrt{\alpha}}{\sqrt{n}}\Big)f-1\Big)\Big\Vert_1\\
&\leq&\sum_{j=1}^n\Big\Vert\Gamma\Big(\frac{1}{\sqrt{n}}\Big)f^{\diamond
j-1}\Big\Vert_1\cdot\Big\Vert\Gamma\Big(\frac{\sqrt{\alpha}}{\sqrt{(1-\alpha)
n}}\Big)f-1\Big\Vert_1\\
&=&\Big\Vert\Gamma\Big(\frac{\sqrt{\alpha}}{\sqrt{(1-\alpha)
n}}\Big)f-1\Big\Vert_1\cdot\sum_{j=1}^n\Big\Vert\Gamma\Big(\frac{1}{\sqrt{n}}\Big)f^{\diamond
j-1}\Big\Vert_1.
\end{eqnarray*}
Observe that employing once again inequality (\ref{HYL}) we can bound the last sum as
\begin{eqnarray*}
\sum_{j=1}^n\Big\Vert\Gamma\Big(\frac{1}{\sqrt{n}}\Big)f^{\diamond
j-1}\Big\Vert_1&\leq&
\sum_{j=1}^n\Big\Vert\Gamma\Big(\frac{\sqrt{j-1}}{\sqrt{n}}\Big)f\Big\Vert_1^{
j-1}\\
&\leq&\sum_{j=1}^n\Vert f\Vert_1^{
j-1}\\
&=& n.
\end{eqnarray*}
Here we are using the fact that $f$ is a density function (in particular is non negative and with integral with respect to $\mu$ equal to one). Therefore
\begin{eqnarray}\label{last}
\Big\Vert\Gamma\Big(\frac{\sqrt{\alpha}}{\sqrt{n}}\Big)f^{\diamond
n}-1\Big\Vert_1&\leq&\Big\Vert\Gamma\Big(\frac{\sqrt{\alpha}}{\sqrt{(1-\alpha)n}}\Big)f-1\Big\Vert_1\cdot\sum_{j=1}^n\Big\Vert\Gamma\Big(\frac{1}{\sqrt{n}}\Big)f^{\diamond
j-1}\Big\Vert_1\nonumber\\
&\leq& n\cdot\Big\Vert\Gamma\Big(\frac{\sqrt{\alpha}}{\sqrt{(1-\alpha)
n}}\Big)f-1\Big\Vert_1.
\end{eqnarray}
Since we are assuming $f$ to be in $\mathcal{L}^p(W,\mu)$ for some
$p> 1$, by the Nelson hyper-contractive property (\ref{nelson}) there exists a
$\gamma\geq 1$ such that
$\Gamma\Big(\frac{1}{\sqrt{\gamma}}\Big)f\in\mathcal{L}^2(W,\mu)$.
Hence, choosing $n$ big enough to ensure that
$\frac{\alpha\gamma}{(1-\alpha) n}\leq 1$  we can write
\begin{eqnarray*}
\Big\Vert\Gamma\Big(\frac{\sqrt{\alpha}}{\sqrt{(1-\alpha)
n}}\Big)f-1\Big\Vert_1&=&\Big\Vert\Gamma\Big(\frac{\sqrt{\alpha\gamma}}{\sqrt{(1-\alpha)
n}}\Big)\Gamma\Big(\frac{1}{\sqrt{\gamma}}\Big)f-1\Big\Vert_1\\
&\leq&\Big\Vert\Gamma\Big(\frac{\sqrt{\alpha\gamma}}{\sqrt{(1-\alpha)
n}}\Big)\Gamma\Big(\frac{1}{\sqrt{\gamma}}\Big)f-1\Big\Vert_2\\
&=&\Big(\sum_{k\geq 3}k!\Big(\frac{\alpha\gamma}{(1-\alpha)
n}\Big)^k\Vert\mathfrak{h}_k\Vert_{H^{\otimes k}}^2\Big)^{\frac{1}{2}},
\end{eqnarray*}
where the $\mathfrak{h}_k$'s are the kernels in the Wiener-It\^o chaos
decomposition of $\Gamma\Big(\frac{1}{\sqrt{\gamma}}\Big)f$. Recall that the
assumptions on the densities of the random variables $X_k$'s and properties of the Ornstein-Uhlenbeck semigroup imply that the chaos expansion of
$\Gamma\Big(\frac{1}{\sqrt{\gamma}}\Big)f$ does not contain chaoses of the first and second orders.\\
Inserting the last estimate in (\ref{last}) we get
\begin{eqnarray}\label{rate}
\Big\Vert\Gamma\Big(\frac{\sqrt{\alpha}}{\sqrt{n}}\Big)f^{\diamond
n}-1\Big\Vert_1&\leq&
n\cdot\Big\Vert\Gamma\Big(\frac{\sqrt{\alpha}}{\sqrt{(1-\alpha)
n}}\Big)f-1\Big\Vert_1 \nonumber\\
&\leq&n\cdot\Big(\sum_{k\geq 3}k!\Big(\frac{\alpha\gamma}{(1-\alpha)
n}\Big)^k\Vert\mathfrak{h}_k\Vert_{H^{\otimes k}}^2\Big)^{\frac{1}{2}}\nonumber \\ 
&\leq&n\Big(\frac{\alpha\gamma}{(1-\alpha)n}\Big)^{\frac{3}{2}}\Big(\sum_{k\geq
3}k!\Vert\mathfrak{h}_k\Vert_{H^{\otimes k}}^2\Big)^{\frac{1}{2}}\nonumber\\ 
&=&\frac{1}{\sqrt{n}}\Big(\frac{\alpha\gamma}{1-\alpha}\Big)^{\frac{3}{2}}\Big(\sum_{k\geq
3}k!\Vert\mathfrak{h}_k\Vert_{H^{\otimes k}}^2\Big)^{\frac{1}{2}}
\end{eqnarray}
(recall that we are assuming $n\geq\frac{\alpha\gamma}{1-\alpha}$). The last series, being equal to $\Big\Vert \Gamma\Big(\frac{1}{\sqrt{\gamma}}\Big)f\Big\Vert_2^2-1$, is convergent; we can therefore pass to the limit as $n$ tends to infinity and obtain the desired result.
\end{proof}

\subsection{The finite dimensional case: a dimension independent Berry-Esseen bound}

Our main result, Theorem \ref{main theorem}, provides a local limit theorem for independent and identically distributed random variables taking values on an abstract Wiener space $(W,H,\mu)$. Observe that for any $d\in\mathbb{N}$ the Euclidean space $\mathbb{R}^d$ together with a standard $d$-dimensional Gaussian measure $\mu$ is an example of such a space (in this case $W=H=\mathbb{R}^d$); therefore, the conclusion of Theorem \ref{main theorem} remains valid in this finite dimensional framework.\\ 
We now want to analyze this particular case in some detail. Let $\{X_n\}_{n\geq 1}$ be a sequence of independent and identically distributed $d$-dimensional random vectors. Assume that the common law of the $X_n$'s is absolutely continuous with respect to $\mu$ with a density $g$ belonging to $\mathcal{L}^2(\mathbb{R}^d,\mu)$ (as before we can replace the exponent $2$ with $p>1$ and use Nelson's estimate).  In Theorem \ref{main theorem} we assumed that: \\

\noindent $i)$\quad $g$ is of the form $\Gamma(\sqrt{\alpha})f$ for some $\alpha\in ]0,1[$ and a non negative $f$ in $\mathcal{L}^2(\mathbb{R}^d,\mu)$;\\

\noindent $ii)$\quad $g$ is orthogonal to the first and second Wiener chaoses.\\

\noindent These two conditions are equivalent respectively to \\   

\noindent $i')$\quad $g$ is of the form 
\begin{eqnarray*}
g(x)=\int_{\mathbb{R}^d}f(\sqrt{\alpha}x+\sqrt{1-\alpha}y)d\mu(y),\quad x\in\mathbb{R}^d
\end{eqnarray*}
for some $\alpha\in ]0,1[$ and a non negative $f$ in $\mathcal{L}^2(\mathbb{R}^d,\mu)$;\\

\noindent $ii')$\quad $E[X_n]=0$ and the covariance matrix of the vector $X_n$ is the identity matrix.\\

\noindent The equivalence between $i)$ and $i')$ is known in the literature as the Mehler's formula (e.g. see Janson (1997)). Concerning the second equivalence, observe that the sets of functions 
\begin{eqnarray*}
\{h_j(x)=x_j,\quad j=1,...,d\}\quad\mbox{ and }\quad\{l_{ij}(x)=x_i x_j-\delta_{ij},\quad i,j=1,...,n\}
\end{eqnarray*}
 constitute an orthogonal basis for the first and second homogeneous chaoses, respectively. Therefore, if $g$ satisfies $ii)$ then
\begin{eqnarray*}
E[X^i_n]&=&\int_{\mathbb{R}^d}x_ig(x)d\mu(x)\\
&=&0,\quad\mbox{ for all }i=1,...,d
\end{eqnarray*} 
and
\begin{eqnarray*}
Cov(X_n^i,X_n^j)&=&E[X_n^iX_n^j]-E[X_n^i]E[X_n^j]\\
&=&\int_{\mathbb{R}^d}x_i x_jg(x)d\mu(x)\\
&=&\int_{\mathbb{R}^d}\delta_{ij}g(x)d\mu(x)\\
&=&\delta_{ij}\quad\mbox{ for all }i,j=1,...,d
\end{eqnarray*}
which corresponds to $ii')$ (the converse is clearly also true).\\
Recall in addition that the total variation distance between two probability measures on $\mathbb{R}^d$, say $\nu_1$ and $\nu_2$, is defined by
\begin{eqnarray*}
d_{TV}(\nu_1,\nu_2):=\sup_{A\in\mathcal{B}(\mathbb{R}^d)}|\nu_1(A)-\nu_2(A)|;
\end{eqnarray*}
moreover, if $\nu_1$ and $\nu_2$ are absolutely continuous with respect to $\mu$ with densities $f_1$ and $f_2$ respectively, then one has
\begin{eqnarray*}
d_{TV}(\nu_1,\nu_2)=\frac{1}{2}\int_{\mathbb{R}^d}|f_1(x)-f_2(x)|d\mu(x).
\end{eqnarray*}
With this notation at hand and following the preceding discussion, we can reformulate Theorem \ref{main theorem} as follows.
\begin{corollary}\label{corollary}
Let $\{X_n\}_{n\geq 1}$ be a sequence of independent and identically distributed $d$-dimensional random vectors. Assume that the common law of the $X_n$'s is absolutely continuous with respect to $\mu$ with a density $g$ belonging to $\mathcal{L}^2(\mathbb{R}^d,\mu)$. If conditions $i')$ and $ii')$ from above are satisfied, then for $n\geq\frac{\alpha}{1-\alpha}$ one has
\begin{eqnarray}\label{total variation}
d_{TV}(\nu_{S_n},\mu)\leq \frac{1}{2}\frac{1}{\sqrt{n}}\Big(\frac{\alpha}{1-\alpha}\Big)^{\frac{3}{2}}(\Vert f\Vert_2^2-1)^{\frac{1}{2}}
\end{eqnarray}
where $\nu_{S_n}$ denotes the law of $S_n:=\frac{X_1+\cdot\cdot\cdot+X_n}{\sqrt{n}}$.
\end{corollary}
\begin{proof}
Inequality (\ref{total variation}) follows from (\ref{rate}) where we can take $\gamma=1$ since $f\in\mathcal{L}^2(\mathbb{R}^d,\mu)$.
\end{proof}
Inequality (\ref{total variation}) provides a Berry-Esseen type bound which depends on $\alpha$, on the second moment of the density $f$ but not on the dimension $d$. This is in contrast to a series of known results where the right hand side of (\ref{total variation}) depends on $d$. More precisely, Bentkus (2003) (see also the references quoted there for earlier results) proves under the assumption $ii')$ from above and the finiteness of $\beta:=E[\Vert X_n\Vert^3]$ (here $\Vert\cdot\Vert$ is the $d$-dimensional Euclidean norm) the inequality
\begin{eqnarray}\label{bentkus}
\sup_{A\in\mathcal{C}}|\nu_{S_n}( A)-\mu(A)|\leq 400\cdot\beta\cdot\frac{d^{\frac{1}{4}}}{\sqrt{n}}
\end{eqnarray}
where $\mathcal{C}$ is the class of convex sets. The bound in (\ref{bentkus}) contains the best known dependence on the dimension under those assumptions. Our Corollary \ref{corollary} requires more stringent conditions, namely $i')$, but has the advantage of being dimension independent.\\

\noindent {\textbf{Acknowledgments}}

We are grateful to Professor Youri Davydov for his fruitful comments and pointing our attention to the references Bloznelis (2002) and Davydov (1992).

Barron, A.R. (1986). Entropy and the central limit theorem. \emph{Ann. Probab.}, 14, 336-342.\\

Bentkus, V. (2003). On the dependence of the Berry-Esseen bound on dimension. \emph{J. Statist. Plann Inference}, 113, 385-402.\\

Bloznelis, M. (2002). A note on the multivariate local limit theorem. \emph{Statistics and Probability Letters}, 59, 227-233.\\

Bogachev, V.I. (1998). \emph{Gaussian Measures}. American Mathematical Society, Providence.\\

Da Pelo, P., Lanconelli A. and Stan, A.I. (2011). A H\"older-Young-Lieb
inequality for norms of Gaussian Wick products. \emph{Inf. Dim.
Anal. Quantum Prob. Related Topics}, 14, 375-407.\\

Da Pelo, P., Lanconelli, A. and Stan, A.I. (2013). An It\^o formula for a
family of stochastic integrals and related Wong-Zakai theorems.
\emph{Stochastic Processes and their Applications}, 123, 3183-3200.\\

Davydov, Y. (1992). A variant of an
infinite-dimensional local limit theorem. \emph{Journal of Soviet Mathematics [1]}, 61, 1853-1856.\\

Gnedenko, B.V. (1954). Local limit theorem for densities. \emph{Doklady Akad. Nauk SSSR}, 95, 5-7.\\

Holden, H., {\O}ksendal, B., Ub{\o}e, J. and Zhang, T.-S. (2009).
\emph{Stochastic Partial Differential Equations- A Modeling, White
Noise Functional Approach, II edition}. Springer, New York.\\

Janson, S. (1997). \emph{Gaussian Hilbert spaces}. Cambridge Tracts in
Mathematics, 129. Cambridge University Press, Cambridge.\\

Lanconelli, A. and Sportelli, L. (2012). Wick calculus for the square of a Gaussian random variable with application to Young and hypercontractive inequalities. \emph{Inf. Dim.
Anal. Quantum Prob. Related Topics} 15, 16 pages.\\

Lanconelli, A. and Stan, A.I. (2010). Some norm inequalities for Gaussian
Wick Products. \emph{Stochastic Analysis and Applications}, 28, 523-539.\\

Lanconelli, A. and Stan, A.I. (2013). A H\"older inequality for norms of Poissonian Wick products. \emph{Inf. Dim.
Anal. Quantum Prob. Related Topics}, 16, 39 pages.\\

Linnik, Y.V. (1973). An information-theoretic proof of the central limit theorem with the Lindberg condition. \emph{Theory Probab. Appl.}, 4, 288-299.\\

Nelson, E. (1973). The free Markoff field. \emph{J. Functional Analysis}, 12, 211-227.\\

Nualart, D. (2006). \emph{Malliavin calculus and Related Topics, II edition}. Springer, New York.\\

Prokhorov, Y.V. (1952). On a local limit theorem for densities. \emph{Doklady Akad. Nauk SSSR}, 83, 797-800.\\

Ranga Rao, R. and Varadarajan, V.S. (1960). A limit theorem for densities. {\em Sankhya}, 22, 261-266.

\end{document}